\documentclass[a4paper,10pt,draft]{article}
\usepackage[english]{babel}
\usepackage{t1enc}
\usepackage{amsfonts}
\usepackage{amsmath,amsthm,amssymb,esint}
\usepackage{color}
\usepackage{authblk}
\usepackage{tikz, pgffor}
\usetikzlibrary{calc,through,backgrounds}
\newtheorem{thm}{Theorem}[section]
\newtheorem{coro}[thm]{Corollary}
\newtheorem{propo}[thm]{Proposition}
\newtheorem{lem}[thm]{Lemma}

\theoremstyle{remark}

\theoremstyle{definition}

\newcommand{\imply}{\ensuremath{\rightarrow}}

\newcommand{\embeds}{\ensuremath{\hookrightarrow}}
\newcommand{\yy}{\ensuremath{\wedge}}
\newcommand{\oo}{\ensuremath{\vee}}

\newcommand{\RR}{\ensuremath{\mathbb{R}}}
\newcommand{\HH}{\ensuremath{\mathcal{H}}}

\newcommand{\menos}{\symbol{92}}
\newcommand{\diam}{\ensuremath{\text{diam}}}

\allowdisplaybreaks 

\begin{document}

\title{On Newton-Sobolev spaces}

\vskip 0.3 truecm

\author{Miguel Andrés Marcos \thanks{The author was supported by Consejo Nacional
de Investigaciones Científicas y Técnicas, Agencia Nacional de Promoción Científica y Tecnológica and Universidad
Nacional del Litoral.\newline \indent Keywords and phrases:
Newton-Sobolev spaces, Spaces of homogeneous type, Poincaré inequality, upper gradients
\newline \indent 2010 Mathematics Subject Classification: Primary
43A85.\newline }}
\affil{\footnotesize{Instituto de Matemática Aplicada del Litoral (CONICET-UNL)\\ Departamento de Matemática (FIQ-UNL)}}

\date{\vspace{-0.5cm}}

\maketitle 

\begin{abstract}
Newton-Sobolev spaces, as presented by N. Shanmugalingam, describe a way to extend Sobolev spaces to the metric setting via upper gradients, for metric spaces with `sufficient' paths of finite length. Sometimes, as is the case of parabolic metrics, most curves are non-rectifiable. As a course of action to overcome this problem, we generalize some of these results to spaces where paths are not necessarily measured by arc length. In particular, we prove the Banach character of the space and the absolute continuity of these Sobolev functions over curves. Under the assumption of a Poincaré-type inequality and an arc-chord property here defined, we obtain the density of some Lipschitz classes, relate Newton-Sobolev spaces to those defined by Haj\symbol{170}asz by means of Haj\symbol{170}asz gradients, and we also get some Sobolev embedding theorems. Finally, we illustrate some non-standard settings where these conditions hold, specifically by adding a weight to arc-length and specifying some conditions over it.
\end{abstract}

\setlength{\parskip}{10pt}

\section{Introduction}

If $\Omega$ is an open set in $\RR^n$ and $f$ is a smooth function defined on $\Omega$, the Fundamental Theorem of Calculus for line integrals implies that for every piecewise smooth path $\gamma$ in $\Omega$ with endpoints $x,y$ we get
\begin{align*}|f(x)-f(y)|\leq \int_\gamma |\nabla f|d|s|.\end{align*}
Nonnegative functions defined in $\Omega$ that satisfy this inequality for every $x,y$ and every $\gamma$ joining them in place of $|\nabla f|$ are referred to as \textsl{upper gradients} (see for example \cite{HeK}).

In the case $\Omega=\RR^n$, one can consider only segments parallel to the coordinate axes instead of more general paths, and those are sufficient to describe partial derivatives and through them gradients. The same can be done if we consider a rotation of these segments, as the Euclidean metric is invariant under rotations, and the same holds for path length. This is not true in a more general setting such as $\RR^2$ with the parabolic metric defined further along this section, where only horizontal segments are rectifiable.

In \cite{Sh}, N. Shanmugalingam describes, via upper gradients, a way to characterize Sobolev spaces $W^{1,p}$ in open sets of $\RR^n$ that extends to metric measure spaces, defining Newton-Sobolev spaces $N^{1,p}$. If the space has `sufficient' rectifiable paths (in the sense that the set of rectifiable paths has nonzero $p$-modulus), an interesting theory of Sobolev functions can be developed, but if the set of rectifiable paths is negligible, this `Sobolev space' is just $L^p$.

Easy enough examples of metric measure spaces with no paths of dimension 1 can be constructed. For instance, take $X=\RR$ with $d(x,y)=|x-y|^{1/2}$, and we get that paths are either 0-dimensional (trivial paths) or 2-dimensional. While `classical' Newton-Sobolev theory in such a space would be nonsensical, a good theory could be developed if we measured path `length' by Hausdorff 2-dimensional measure $\HH^2$ with respect to the new distance $d$. Of course, $\HH^2_d$ coincides with $\HH^1$ with respect to the Euclidean distance, and the above example seems to be just a change of parameters. 

In a more interesting scenario, we consider parabolic metrics associated to a matrix, see for instance \cite{Gu}. Take an $n\times n$ diagonal matrix $D$ with eigenvalues $\alpha_1,\ldots, \alpha_n\geq 1$. For $x\in \RR^n$ and $\lambda>0$, we define
\begin{align*}T_\lambda x=e^{D\log\lambda}x=\left(
\begin{array}{ccc}\lambda^{\alpha_1} & & \makebox(0,0){\text{\LARGE{0}}}\\& \ddots\\\text{\LARGE{0}}& & \lambda^{\alpha_n}\end{array}\right)\left(
\begin{array}{c}x_1\\ \vdots\\ x_n\end{array}\right).\end{align*}

For a norm $\|\cdot\|$ in $\RR^n$ it can be shown that for $x\neq 0$, $\|T_\lambda x\|$ is continuous, strictly increasing in $\lambda$, tends to 0 as $\lambda\imply 0$ and tends to $\infty$ as $\lambda\imply\infty$. Then there exists a unique $0<\rho(x)<\infty$ such that $\|T_{1/\rho(x)}x\|=1$. If we define
\begin{align*}d(x,y)=\rho(x-y)\end{align*}
for $x\neq y$ and $d(x,x)=0$, then $d$ is a traslation invariant metric that also satisfies $d(T_\lambda x,T_\lambda y)=\lambda d(x,y)$ and $d(x,y)=1$ iff $|x-y|=1$, $d(x,y)<1$ iff $|x-y|<1$, $d(x,y)>1$ iff $|x-y|>1$. These metrics thus defined can have different Hausdorff dimensions, see \cite{A}.

The word parabolic refers to the case $\alpha_1=\ldots=\alpha_{n-1}=1$ and $\alpha_n=2$, which provides the right dilations for the heat equation and other partial differential equations of parabolic type (see \cite{Fa}). For example, if we consider $\RR^2$ with $D=\left(
\begin{array}{cc}1 & 0\\ 0 & 2\end{array}\right)$ and the maximum norm, we obtain
\begin{align*}d\left((x,y),(x',y')\right)=\max\left\{|x-x'|,|y-y'|^{1/2}\right\},\end{align*}
and it can be shown balls have Hausdorff dimension 3 (in fact they are Ahlfors 3-regular). Here, the only non-trivial rectifiable paths are horizontal segments, so even though \textsl{there are} rectifiable paths, the space is not connected by them. Smooth non-horizontal paths have Hausdorff dimension 2, so we see that this measure is not rotation invariant.

As another example of heterogeneity, we can consider adding a weight $\omega$ to arc-length by using the measure $d\mu=\omega d\HH^1$. In this case this path measure will not necessarily be invariant under any kind of isometry.

In this work, following the ideas in \cite{Sh}, we develop a more general theory of Newton-Sobolev spaces by replacing Hausdorff 1-dimensional measure by an arbitrary measure $\mu$ as a way of measuring path `lengths'.

In sections 2 and 3 we generalize all the machinery needed to construct Newton-Sobolev spaces. In section 4 we define these spaces and prove they are complete. In section 5 we call for some additional properties, such as Poincaré inequality, needed to prove some more interesting results, as Lipschitz density or Sobolev embeddings. We also compare Newton-Sobolev spaces with another kind of Sobolev space in metric spaces: Haj\symbol{170}asz-Sobolev spaces.

\section{$\mu$-arc length and upper gradients}

Classical definitions of arc length, length function, arc length parametrization and line integrals in the metric setting can be found in \cite{He}. In this section we modify these concepts so they apply in more general ways to measure path `lengths'.

Given a metric space $(X,d)$ and a (compact) path $\gamma:[a,b]\imply X$, i.e. a continuous function from $[a,b]$ into $X$, its length is defined as
\begin{align*}l(\gamma)=\sup_{(t_i)_i}\sum_i d(\gamma(t_i),\gamma(t_{i+1})),\end{align*}
where the supremum is taken over all partitions of $[a,b]$. We say that $\tilde{\gamma}$ is a sub-path of $\gamma$ if it is the restriction of of $\gamma$ to a subinterval of $[a,b]$. We say that a path (or subpath) is trivial if it is a constant path (for injective paths this means $a=b$).

The concept of arc length of a path is similar to, but not equal to, Hausdorff one-dimensional measure $\HH^1$ of its image, but they do coincide for injective paths (see \cite{Fl}). From this result, for injective paths and for Borel nonnegative measurable functions we get that
\begin{align*}\int_\gamma g ds= \int_{ Im(\gamma)}gd\HH^1,\end{align*}
where $d\sigma$ is arc-length, and from this we can think of exchanging the measure $\HH^1$ for another Borel measure, as $\HH^s$.

Let $\mu$ be a non-atomic Borel measure in $X$ (in the sense that $\mu(\{x\})=0$ for each $x\in X$). Define $\Gamma^\mu$ as the set of all non trivial injective paths $\gamma$ in $X$ such that $0<\mu( Im(\tilde{\gamma}))<\infty$ for all non trivial subpaths of $\gamma$. For nonnegative Borel functions $g:X\imply[0,\infty]$ we define
\begin{align*}\int_\gamma g = \int_{ Im(\gamma)}gd\mu.\end{align*}

Now, for a path $\gamma:[a,b]\imply X$ in $\Gamma^\mu$, we define $h(\gamma)=\mu( Im(\gamma))$ and its \textbf{$\mu$-arc length} $\nu_\gamma:[a,b]\imply\RR$ as
\begin{align*}\nu_\gamma(x)=h(\gamma|_{[a,x]}).\end{align*}

\begin{lem}For paths $\gamma:[a,b]\imply X$ in $\Gamma^\mu$, we have that $\nu_\gamma$ is strictly increasing, continuous, onto $[0,h(\gamma)]$, and besides
\begin{align*}h(\gamma)=h(\gamma|_{[a,x]})+h(\gamma|_{[x,b]}).\end{align*}\end{lem}
\begin{proof}
$\nu_\gamma$ is clearly increasing. Continuity follows from $\mu$ being non-atomic, and surjectivity follows from it being continuous and increasing.
The fact that $\nu_\gamma$ is strictly increasing follows from the fact that every non trivial subcurve of $\gamma$ has positive measure, as $\gamma\in \Gamma^\mu$.\end{proof}

\begin{thm}For $\gamma:[a,b]\imply X$ in $\Gamma^\mu$, there is a unique $\gamma_h:[0,h(\gamma)]\imply X$ such that
\begin{align*}\gamma=\gamma_h\circ\nu_\gamma,\end{align*}
$ Im(\gamma)=Im(\gamma_h)$ and $\nu_{(\gamma_h)}(t)=t$ in $[0,h(\gamma)]$ (therefore $\gamma_h=\gamma_h\circ\nu_{\gamma_h}$). We call this the \textbf{$\mu$-arc length parametrization} of $\gamma$.\end{thm}
\begin{proof} As $\nu_\gamma:[a,b]\imply[0,h(\gamma)]$ is strictly increasing and onto, it is a bijection between $[a,b]$ and $[0,h(\gamma)]$ and we can define
\begin{align*}\gamma_h=\gamma\circ\nu^{-1}_\gamma.\end{align*}
We immediately see that $ Im(\gamma)=Im(\gamma_h)$, and
\begin{align*}\nu_{(\gamma_h)}(t) &=\mu(\gamma_h([0,t]))=\mu(\gamma(\nu^{-1}_\gamma([0,t])))=\mu(\gamma([a,\nu^{-1}_\gamma(t)]))\\
&=\nu_\gamma(\nu^{-1}_\gamma(t))=t.\end{align*}\end{proof}

\begin{thm}If $\gamma:[0,h]\imply X$ is a path in $\Gamma^\mu$ parametrized by $\mu$-arc length, then for every Borel set $B$ of $[0,h]$, we have 
\begin{align*}\mu(\gamma(B))=l(B).\end{align*}
Furthermore, if $g:X\imply\RR$ is nonnegative and Borel measurable, then for each subpath $\tilde{\gamma}=\gamma|_{[a,b]}$ we have
\begin{align*}\int_{\tilde{\gamma}}g=\int_a^b g\circ\tilde{\gamma}.\end{align*}\end{thm}

Finally, we get the same result as with rectifiable curves.

\begin{thm}\label{ac}Given a function $f:X\imply\RR$ and a path $\gamma:[0,h]\imply X$ in $\Gamma^\mu$ parametrized by $\mu$-arc length, if there exists a Borel measurable nonnegative $\rho:X\imply\RR$ satisfying
\begin{align*}|f(\gamma(s))-f(\gamma(t))|\leq\int_{\gamma|_{[s,t]}}\rho<\infty\end{align*}
for every $0\leq s<t\leq h$, then $f\circ\gamma:[0,h]\imply\RR$ is absolutely continuous.\end{thm}
\begin{proof}
Let $\epsilon>0$. As $\rho\in L^1( Im(\gamma),\mu)$, by absolute continuity of the integral there exists $\delta>0$ such that for every $E\subset Im(\gamma)$ with $\mu(E)<\delta$ we have $\int_E\rho d\mu<\epsilon$. Then if $0\leq a_1<b_1<a_2<b_2<\ldots<a_n<b_n\leq h$ satisfy $\sum_i |b_i-a_i|<\delta$,
\begin{align*}\mu(\cup_i\gamma([a_i,b_i]))=\sum_i\nu_\gamma(b_i)-\nu_\gamma(a_i)=\sum_i b_i-a_i<\delta\end{align*}
and therefore
\begin{align*}\sum_i |f\circ\gamma(b_i)-f\circ\gamma(a_i)|\leq\sum_i \int_{\gamma|_{[a_i,b_i]}}\rho=\int_{\cup_i\gamma([a_i,b_i])}\rho d\mu<\epsilon.\end{align*}\end{proof}

Let now $\Gamma^*$ be a subset of $\Gamma^\mu$, closed under taking subpaths (i.e. if $\gamma\in\Gamma^*$ and $\tilde{\gamma}$ is a non-trivial subpath of $\gamma$, then $\tilde{\gamma}\in\Gamma^*$). A nonnegative Borel measurable function $\rho$ satisfying 
\begin{align*}|f(x)-f(y)|\leq \int_\gamma \rho\end{align*}
for every $\gamma\in \Gamma^*$ with endpoints $x,y$, for every pair of points $x,y$ with $f(x),f(y)$ finite is called a \textbf{$\mu$-upper gradient} for $f$ with respect to $\Gamma^*$. As theorem \ref{ac} shows, if a function $f$ has an upper gradient with respect to $\Gamma^*$ that is integrable over each path in $\Gamma^*$, then it is absolutely continuous over every path in $\Gamma^*$.

Let $\RR^2$ be equiped with the parabolic distance $d$ discussed in the introduction, and let $\mu=\HH^2$. If $\gamma$ is a segment joining $x=(a,ka+b)$ with $y=(a+h,k(a+h)+b)$ for some $h>0$, then its measure $\mu$ is just its height $|k|h$, while its length is $\sqrt{1+k^2}h$ so in fact we have $d\mu=\frac{k}{\sqrt{1+k^2}}dl$ over these paths (clearly when $k\imply 0$ we get $\mu=0$ and when $k\imply\infty$, $\mu=l$). 

Now, for $f$ smooth,
\begin{align*}|f(y)-f(x)| &\leq\int_0^{\sqrt{1+k^2}h} \left|\nabla f\left(a+\frac{t}{\sqrt{1+k^2}},b+\frac{t}{\sqrt{1+k^2}}\right)\right|dt\\
&= \int_\gamma|\nabla f|ds\\
&=\frac{\sqrt{1+k^2}}{k}\int_{Im(\gamma)}|\nabla f|d\mu,\end{align*}
and the same bound can be shown in a similar way for $h<0$.

Therefore if we consider $\Gamma^*_k$ to be the set of all polygonal paths made up of segments of slope $\pm k$ for a fixed $0<k<\infty$, we obtain that $\frac{\sqrt{1+k^2}}{k}|\nabla f|$ is an upper gradient for $f$ with respect to $\Gamma^*_k$. The following picture illustrates a path of $\Gamma_k^*$ for $k=1$.

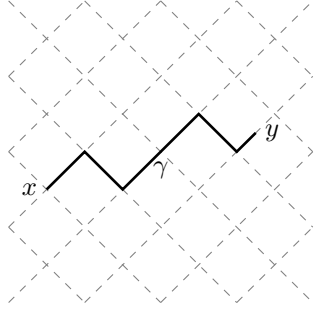
\begin{center}
\begin{tikzpicture}

\draw[gray, dashed](-2,1)--(-1,2);
\draw[gray, dashed](-2,0)--(0,2);
\draw[gray, dashed](-2,-1)--(1,2);
\draw[gray, dashed](-2,-2)--(2,2);
\draw[gray, dashed](-1,-2)--(2,1);
\draw[gray, dashed](0,-2)--(2,0);
\draw[gray, dashed](1,-2)--(2,-1);

\draw[gray, dashed](-2,-1)--(-1,-2);
\draw[gray, dashed](-2,0)--(0,-2);
\draw[gray, dashed](-2,1)--(1,-2);
\draw[gray, dashed](-2,2)--(2,-2);
\draw[gray, dashed](-1,2)--(2,-1);
\draw[gray, dashed](0,2)--(2,0);
\draw[gray, dashed](1,2)--(2,1);

\draw[line width=1pt] (-1.5,-0.5)--(-1,0)--(-0.5,-0.5)--(0.5,0.5)--(1,0)--(1.25,0.25);
\node[left]  at (-1.5,-0.5){$x$};
\node[right]  at (1.25,0.25){$y$};
\node[below]  at (0,0){$\gamma$};

\end{tikzpicture}
\end{center}

Now, if we consider $X=\RR^n$ with Euclidean distance, but $d\mu=\omega d\HH^1$ where $\omega$ and $\frac{1}{\omega}$ are locally integrable with respect to $\HH^1$, we obtain $\Gamma^\mu=\Gamma^{rect}$, where $\Gamma^{rect}$ is the set of all non-trivial injective rectifiable paths. For $f$ smooth and $\gamma \in \Gamma^\mu$,
\begin{align*}|f(y)-f(x)| &\leq \int_\gamma|\nabla f|=\int_{Im(\gamma)}|\nabla f|d\HH^1\\
&=\int_{Im(\gamma)}\frac{|\nabla f|}{\omega}d\mu,\end{align*}
and in fact the same can be applied to any `classical' upper gradient of a function $f$. There is clearly a one to one correspondence between upper gradients $\rho$ with $\HH^1$ and upper gradients of the form $\rho/\omega$ with measure $\mu$.

\section{Modulus of a path family and $p$-weak upper gradients}

Let now $m$ be a Borel measure on $X$. As in \cite{Sh}, we adjust the definition of modulus of a set of measures in \cite{Fu} to path families.

For every family $\Gamma\subset\Gamma^\mu$ and $0<p<\infty$, we define its \textbf{$p$-modulus} as
\begin{align*}Mod_p(\Gamma)=\inf\int_X g^p dm\end{align*}
where the infimum is taken over all nonnegative Borel measurable functions $g:X\imply\RR$ satisfying
\begin{align*}\int_\gamma g\geq1\end{align*}
for every $\gamma\in\Gamma$.

The following results can be found in \cite{Fu}, we state them here in the language of paths instead of measures.

\begin{thm}$Mod_p$ is an outer measure on $\Gamma^\mu$.\end{thm}
%\begin{proof}The fact that $Mod_p(\emptyset)=0$ and the monotonicity are immediate. For $\sigma$-subaditivity, if $\Gamma=\cup_i \Gamma_i$, given $\epsilon>0$ we take $g_i$ with $\int_\gamma g_i\geq 1$ for every $\gamma\in\Gamma_i$ and such that
%\begin{align*}\int_X g_i^pdm\leq Mod_p(\Gamma_i)+2^{-i}\epsilon\end{align*}
%Now, if $g=\sup_i g_i$, $g$ satisfies $\int_\gamma g\geq 1$ for every $\gamma\in\Gamma$, and
%\begin{align*}Mod_p(\Gamma)\leq\int_X g^pdm\leq\sum_i\int_X g_i^pdm\leq \sum_i Mod_p(\Gamma_i)+\epsilon.\end{align*}\end{proof}

As expected, we say that a property holds for \textsl{$p$-almost every path} $\gamma\in\Gamma^\mu$ if the set $\Gamma$ where it does not hold has $Mod_p(\Gamma)=0$. A useful property of sets of $p$-modulus zero is the following.

\begin{lem}\label{A}$Mod_p(\Gamma)=0$ if and only if there exists a nonnegative Borel measurable function $g$ satisfying $\int_X g^pdm<\infty$ and
\begin{align*}\int_\gamma g=\infty\end{align*}
for every $\gamma\in\Gamma$.\end{lem}
%\begin{proof}For the `if' part, for every $n$, $g_n=\frac{1}{n}g$ satisfies $\int_\gamma g_n=\infty\geq1$ for every $\gamma\in\Gamma$. Then 
%\begin{align*}Mod_p(\Gamma)\leq\int_X g_n^pdm=\frac{1}{n^p}\int_X g^pdm\imply0.\end{align*}
%Now, if $Mod_p(\Gamma)=0$, then for each $n$ we can find $g_n$ satisfying $\int_\gamma g_n\geq 1$ for every $\gamma\in\Gamma$ and $\int_X g_n^pdm<4^{-n}$. Then if we define $g=(\sum_n 2^ng_n^p)^{1/p}$, $g$ is Borel measurable, nonnegative and $\int_X g^pdm=\sum_n 2^n\int_X g_n^pdm\leq1<\infty$, and besides
%\begin{align*}\int_\gamma g\geq\int_\gamma 2^{n/p}g_n\geq2^{n/p}\end{align*}
%for every $n$, therefore $\int_\gamma g=\infty$ for every $\gamma\in\Gamma$.\end{proof}

We also need the following result.

\begin{lem}\label{B}If $\int|g_n-g|^pdm\imply 0$, there exists a subsequence $(g_{n_k})_k$ such that $\int_\gamma|g_{n_k}-g|\imply0$ for $p$-almost every $\gamma\in\Gamma^\mu$.\end{lem}
%\begin{proof}Without loss of generality we assume $g_n\geq0$ and $\int g_n^pdm\imply 0$, and we need to prove $\int_\gamma g_{n_k}\imply 0$ for some subsequence $(g_{n_k})$. We take a subsequence satisfying
%\begin{align*}\int_Xg_{n_k}^pdm<2^{-k(p+1)}.\end{align*}
%Let now $\Gamma_k=\{\gamma: \int_\gamma g_{n_k}\geq 2^{-k}\}$ and $\Gamma=\limsup_k\Gamma_k$. Clearly $\int_\gamma 2^kg_{n_k}\geq1$ for each $\gamma\in\Gamma_k$, and therefore
%\begin{align*}Mod_p(\Gamma_k)\leq \int_X 2^{kp}g_{n_k}^pdm<2^{-k},\end{align*}
%and for every $j$,
%\begin{align*}Mod_p(\Gamma)\leq Mod_p(\cup_{k>j}\Gamma_k)\leq\sum_{k>j}Mod_p(\Gamma_k)<2^{-j}\end{align*}
%and $Mod_p(\Gamma)=0$. Finally, if $\gamma\not\in\Gamma$, there exists $j$ such that for $k>j$, $\int_\gamma g_{n_k}<2^{-k}$ and we have what we wanted.
%\end{proof}

Given a set $E\subset X$ we define
\begin{align*}\Gamma_E=\{\gamma\in\Gamma^\mu: Im(\gamma)\cap E\neq\emptyset\},\end{align*}
\begin{align*}\Gamma_E^+=\{\gamma\in\Gamma^\mu:\mu( Im(\gamma)\cap E)>0\}\end{align*}
and we have the following lemma

\begin{lem}If $m(E)=0$, then $Mod_p(\Gamma_E^+)=0$.\end{lem}
\begin{proof}Trivial, as $g=\infty\chi_E$ satisfies $g=0$ $m$-almost everywhere, but $\int_\gamma g=\infty$ for every $\gamma\in\Gamma_E^+$.\end{proof}

A nonnegative Borel measurable function $\rho$ satisfying 
\begin{align*}|f(x)-f(y)|\leq \int_\gamma \rho\end{align*}
for $p$-almost every $\gamma\in \Gamma^\mu$ with endpoints $x,y$ is called a \textbf{$p$-weak upper gradient} for $f$.

As in Shanmugalingam's case, we do not lose much by restricting ourselves to weak upper gradients.

\begin{propo}If $\rho$ is a $p$-weak upper gradient for $f$ and $\epsilon>0$, there exists an upper gradient $\rho_\epsilon$ for $f$ such that $\rho_\epsilon\geq\rho$ and $\|\rho-\rho_\epsilon\|_p<\epsilon$.\end{propo}
\begin{proof}Let $\Gamma$ be the set of paths where the inequality for $\rho$ does not hold ($Mod_p(\Gamma)=0$). Then there exists $g\geq0$ Borel measurable with $\int_X g^pdm<\infty$ but $\int_\gamma g=\infty$ for every $\gamma\in\Gamma$. We define
\begin{align*}\rho_\epsilon=\rho + \frac{\epsilon}{1+\|g\|_p}g,\end{align*}
it is clear that $\rho_\epsilon\geq\rho$, $\int_\gamma\rho_\epsilon\geq 1$ for every $\gamma$, so $\rho_\epsilon$ is an upper gradient for $f$, and finally
\begin{align*}\|\rho_\epsilon-\rho\|_p=\epsilon\frac{\|g\|_p}{1+\|g\|_p}<\epsilon.\end{align*}\end{proof}

As seen in \ref{ac}, functions with `small' upper gradients are absolutely continuous on curves. We say that a function $f$ is $ACC_p$ or \textbf{absolutely continuous over $p$-almost every path} if $f\circ\gamma_h:[0,h(\gamma)]\imply\RR$ is absolutely continuous for $p$-almost every $\gamma$.

\begin{lem}If a function $f$ has a $p$-weak upper gradient $\rho\in L^p$, it is $ACC_p$.\end{lem}
\begin{proof}Let $\Gamma_0$ be the set of all paths $\gamma$ such that $|f(x)-f(y)|>\int_\gamma\rho$ and let $\Gamma_1$ be the set of all paths with a subpath in $\Gamma_0$. As $\rho$ is a weak upper gradient, $Mod_p(\Gamma_0)=0$, but if $g$ satisfies $\int_\gamma g\geq1$, it also satisfies $\int_{\tilde{\gamma}} g\geq1$ for every subpath $\tilde{\gamma}$ of $\gamma$, and therefore
\begin{align*}Mod_p(\Gamma_1)\leq Mod_p(\Gamma_0)=0.\end{align*}
Let $\Gamma_2$ be the set of all paths $\gamma$ with $\int_\gamma\rho=\infty$. Then as $\rho\in L^p$, $Mod_p(\Gamma_2)=0$.
For paths not in $\Gamma_1\cup\Gamma_2$, we can apply \ref{ac} and we conclude the lemma. \end{proof}

We will also need the following lemma later on.

\begin{lem}\label{ac2}If $f$ is $ACC_p$ and $f=0$ $m$-almost everywhere, then the family
\begin{align*}\Gamma=\{\gamma\in\Gamma^\mu: f\circ\gamma\not\equiv0\}\end{align*}
has $p$-modulus zero.\end{lem}
\begin{proof} Let $E=\{x:f(x)\neq0\}$, then $m(E)=0$ and $\Gamma=\Gamma_E$. As $\Gamma_E^+$ has modulus zero (because $m(E)=0$), we only need to see that $\Gamma_E\menos\Gamma_E^+$ also has modulus zero. But if $\gamma\in \Gamma_E\menos\Gamma_E^+$, $ Im(\gamma)\cap E\neq\emptyset$ but $\mu( Im(\gamma)\cap E)=0$, therefore $\gamma_h^{-1}(E)$ has length 0 in $\RR$ and $f\circ\gamma_h$ is nonzero in a set of length 0, and if $E\neq\emptyset$ this set is not empty and $f\circ\gamma_h$ cannot be absolutely continuous. Therefore $Mod_p(\Gamma_E\menos\Gamma_E^+)=0$.\end{proof}

\section{Extended Newton-Sobolev spaces $N^{1,p}$}

From now on, we will work on a fixed subset $\Gamma^*\subset \Gamma^\mu$, closed under taking subpaths. Properties defined on the previous section, such as $p$-weak upper gradients or $ACC_p$, can be easily adjusted to $\Gamma^*$ instead of $\Gamma^\mu$. We will also require that the space $X$ be connected by paths belonging to $\Gamma^*$. This is the case of $\Gamma^*_k$ in the example of $\RR^2$ with the parabolic metric. For the Euclidean case, it is sufficient to consider piecewise linear paths made of segments parallel to the coordinate axis instead of all rectifiable paths to obtain a theory of Sobolev spaces, but in general this need not be the case.

We define the space $\tilde{N}^{1,p}$ as the space of all functions $f$ having a $p$-weak upper gradient, both with finite $p$-norms. We define the $N^{1,p}$ norm as
\begin{align*}\|f\|_{N^{1,p}}=\|f\|_p+\inf_\rho\|\rho\|_p,\end{align*}
where the infimum is taken over all $p$-weak upper gradients of $f$.

It immediately follows from definition that $(\tilde{N}^{1,p},\|\cdot\|_{N^{1,p}})$ is a semi-normed vector space. Moreover, if $f,g\in\tilde{N}^{1,p}$, then 
\begin{align*}|f|,\min\{f,g\},\max\{f,g\}\in\tilde{N}^{1,p}.\end{align*}
As seen before, every function in $\tilde{N}^{1,p}$ is $ACC_p$.

$\tilde{N}^{1,p}$ is not a normed space, as two distinct functions can be equal almost everywhere, but also because a function may be in $\tilde{N}^{1,p}$ while a function equal almost everywhere to it may not. We do have the following as a corollary of \ref{ac2}.

\begin{coro}If $f,g\in\tilde{N}^{1,p}$ and $f=g$ $m$-a.e., then $\|f-g\|_{N^{1,p}}=0$.\end{coro}

Finally, we define the equivalence relation $f\sim g$ iff $\|f-g\|_{N^{1,p}}=0$, and the quotient space $N^{1,p}=\tilde{N}^{1,p}/\sim$, the \textbf{generalized Newton-Sobolev space}. 

In the case of $\RR^2$ with the parabolic distance, $\mu=\HH^2$ and $m=\HH^3$ defined in the introduction, we saw in section 2 that if we consider $\Gamma^*_k$ for a fixed $k>0$ as our path family, we obtain that $\frac{\sqrt{1+k^2}}{k}|\nabla f|$ is an upper gradient for $f$. In fact, one can show that $N^{1,p}=W^{1,p}$ with equivalent norms. If we consider the whole of $\Gamma^\mu$, this will not happen, as we can see by considering that, if $\rho$ is a bounded upper gradient for $f$, for paths $\gamma$ joining $(x,y)$ and $(x',y')$ we obtain
\begin{align*}|f(x,y)-f(x',y')|\leq \int_\gamma \rho\leq \|\rho\|_\infty \mu(Im(\gamma)).\end{align*}
So for the case of segments in $\Gamma^*_k$ we would obtain
\begin{align*}|f(x,y)-f(x',y')|\leq \int_\gamma \rho\leq \|\rho\|_\infty |y-y'|,\end{align*}
and if we allow segments of arbitrarily small height we obtain that $f$ must be cylindrical, $f(x,y)=g(y)$, so unless $f\equiv 0$ or $p=\infty$, we cannot obtain $f\in L^p(dm)$. For $p=\infty$, we obtain that $N^{1,\infty}$ consists of cylindrical bounded functions $f(x,y)=g(y)$ with $g$ a Lipschitz-1 function (in the Euclidean sense).

Back to the example of $X=\RR^n$ with Euclidean distance and $d\mu=\omega d\HH^1$, where both $\omega$ and $\frac{1}{\omega}$ are locally integrable with respect to $\HH^1$, so for the measure $dm=\omega^p dx$ (where $dx$ is Lebesgue measure) we get that the space $N^{1,p}(dm)$ consists of those $f\in L^p(dm)$ such that $\frac{|\nabla f|}{\omega}\in L^p(dm)$, or in terms of Lebesgue measure,
\begin{align*}N^{1,p}=\{f: \omega f, |\nabla f|\in L^p(dx)\}.\end{align*}
For the particular case $\omega(x)=1+|x|$, this space coincides with the Sobolev-Hermite space $\mathfrak{L}^p_1$, as defined in \cite{BT}.

We will show now, as \cite{Sh}, that $N^{1,p}$ is a Banach space, but first a lemma.

\begin{lem}Let $F\subset X$ be such that
\begin{align*}\inf\left\{\|f\|_{N^{1,p}}:f\in\tilde{N}^{1,p}(X)\yy f|_F\geq1\right\}=0.\end{align*}
Then $Mod_p(\Gamma_F)=0$.\end{lem}
\begin{proof}For every $n$ we take $v_n\in\tilde{N}^{1,p}(X)$ with $v_n|_F\geq1$ and $\|v_n\|_{N^{1,p}}<2^{-n}$, and take weak upper gradients $\rho_n$ of $v_n$ with $\|\rho_n\|_p<2^{-n}$. Take $u_n=\sum_1^n|v_k|,g_n=\sum_1^n\rho_k$ (each $g_n$ will be a weak upper gradient of $u_n$) and $u=\sum|v_n|$ (observe that $u|_F=\infty$), $g=\sum\rho_n$.
Every $u_n$ turns to be in $\tilde{N}^{1,p}$, and $(u_n),(g_n)$ are Cauchy in $L^p$, therefore convergent in $L^p$ to functions $\tilde{u},\tilde{g}$ respectively. Then $u=\tilde{u},g=\tilde{g}$ a.e. and we have $\int|u|^p<\infty$.
Let $E=\{x\in X: u(x)=\infty\}$, then $m(E)=0$ (as $\int_X|u|^p<\infty$) and $F\subset E$.
If we take
\begin{align*}\Gamma=\left\{\gamma:\int_\gamma g=\infty\oo\int_\gamma g_n\not\imply\int_\gamma g\right\}\end{align*}
then $Mod_p(\Gamma)=0$ from \ref{A} and \ref{B}.
If $\gamma\not\in\Gamma\cup\Gamma_E^+$ ($Mod_p(\Gamma_E^+)=0$), then there exists $y\in Im(\gamma)\menos E$, and if $x\in Im(\gamma)$, 
\begin{align*}|u_n(x)|\leq|u_n(y)|+\int_\gamma g_n\leq|u(y)|+\int_\gamma g,\end{align*}
therefore $|u(x)|<\infty$ and $\gamma\not\in\Gamma_E$, and we have 
\begin{align*}Mod_p(\Gamma_F)\leq Mod_p(\Gamma_E)\leq Mod_p(\Gamma\cup\Gamma_E^+)=0.\end{align*}\end{proof}

\begin{thm}$N^{1,p}$ is Banach.\end{thm}
\begin{proof}Let $(u_n)$ be Cauchy in $N^{1,p}$. By taking subsequences we can assume
\begin{align*}\|u_n-u_{n+1}\|_{N^{1,p}}< 2^{-n\frac{p+1}{p}}\end{align*}
and take weak upper gradients $g_n$ of $u_n-u_{n+1}$ with
\begin{align*}\|g_n\|_p<2^{-n}.\end{align*}
Define
\begin{align*}E_n=\{x\in X:|u_n(x)-u_{n+1}(x)|\geq 2^{-n}\},E=\lim\sup E_n.\end{align*}
If $x\not\in E$, then there exists $n_x$ such that $|u_n(x)-u_{n+1}(x)|<2^{-n}$ for $n\geq n_x$ and therefore outside of $E$
\begin{align*}u(x)=\lim u_n(x)\end{align*}
it is well defined.

By Tchebyschev's inequality, $\mu(E_n)\leq 2^{np}\|u_n-u_{n+1}\|_p^p\leq2^{-n}$, and 
\begin{align*}\mu(E)\leq\sum_n^\infty\mu(E_k)\leq 2^{-n}\cdot2,\end{align*}
for every $n$, and on the other hand 
\begin{align*}\inf\left\{\|f\|_{N^{1,p}}:f\in\tilde{N}^{1,p}(X)\yy f|_{E}\geq1\right\} &\leq\\
&\hspace{-3cm}\leq \sum_n^\infty \inf\left\{\|f\|_{N^{1,p}}:f\in\tilde{N}^{1,p}(X)\yy f|_{E_n}\geq1\right\}\\
&\hspace{-3cm}\leq\sum_n^\infty2^{np}\|u_n-u_{n+1}\|^p_{N^{1,p}}\leq 2^{-n}\cdot 2\end{align*}
for every $n$.

By the previous lemma, $Mod_p(\Gamma_E)=0$, and if we define $u|_E\equiv0$, as $(u_n)$ is Cauchy in $L^p$ and $u_n\imply u$ a.e., we have $\int|u|^p<\infty$. Finally for $\gamma\not\in\Gamma_E$ with endpoints $x,y$ we have
\begin{align*}|(u-u_n)(x)-(u-u_n)(y)| &\leq\sum_n^\infty|(u_{k+1}-u_k)(x)-(u_{k+1}-u_k)(y)|
\\&\leq\sum_n^\infty\int_\gamma g_k,\end{align*}
and we get that $\sum_n^\infty g_k$ is a $p$-weak upper gradient of $u-u_n$ (which tends to 0 in $L^p$), and we have $u\in N^{1,p}$ and
\begin{align*}\|u-u_n\|_{N^{1,p}}\leq \|u-u_n\|_p+\|\sum_n^\infty g_k\|_p\imply 0.\end{align*}\end{proof}

\section{Poincaré Inequality}

If there is no relationship between the `space measure' $m$ and the `path measure' $\mu$, most standard results about $N^{1,p}$ cannot be proven. The standard way of relating them is by Poincaré inequality. In our case we will also need a relationship between the `path measure' and the distance function.

We say that $X$ supports a \textbf{$(1,p)$-Poincaré inequality} of exponent $\beta$ if there exists $C>0, \lambda\geq1$ such that for every ball $B$ and every pair $f,\rho$ defined in $B$ such that $f\in L^1(B)$ and $\rho$ is an upper gradient of $f$ in $B$, we have
\begin{align*}\fint_B|f-f_B|dm\leq C\diam(B)^\beta\left(\fint_{\lambda B}\rho^p dm\right)^{1/p}.\end{align*}

In Shanmugalingam's case, this property suffices for proving that Lipschitz functions are dense in $N^{1,p}$. One crucial fact for proving this is that the length of a path is always greater than or equal to the distance between any pair of points over the curve, but in our context this may not be the case. We say that the family $\Gamma^*$ has the \textbf{$\mu$-arc-chord property} with exponent $\beta$ if there exists $C_\mu>0$ such that for every $\gamma\in\Gamma^*$ (and thus for every subpath of that $\gamma$, as $\Gamma^*$ is closed under taking subpaths), we get that
\begin{align*}\diam(Im(\gamma))^\beta\leq C_\mu\mu(Im(\gamma)).\end{align*}

Observe that the usual chord-arc property (see, for instance, \cite{D}) means the opposite inequality: $l(\gamma)\leq Cd(x,y)$ if $\gamma$ is a path joining $x$ and $y$ (which in turn implies $l(\gamma)\sim d(x,y)$, as the reverse inequality $d(x,y)\leq l(\gamma)$ always holds). We do not require this control over the measure of the curves in $\Gamma^*$, but the opposite one (thus we reverse the word order in the definition).

In this section we will prove some results that arise from these properties, and then we will go back to the example $d\mu=\omega d\HH^1$.

First, we will prove a series of lemmas that will give us sufficient conditions for Lipschitz functions to be dense in $N^{1,p}$.

\begin{lem}Let $f$ be $ACC_p$ such that $f|_F=0$ $m$-a.e., for $F$ a closed subset of $X$. If $\rho$ is an upper gradient of $f$, then $\rho\chi_{X\menos F}$ is a $p$-weak upper gradient of $f$.\end{lem}
\begin{proof}Let $\Gamma_0$ be the set of paths for which $f\circ\gamma_h$ is not absolutely continuous, and let $E=\{x\in F: f(x)\neq0\}$, so $Mod_p(\Gamma_0\cup\Gamma_E^+)=0$. Now, if $\gamma\not\in \Gamma_0\cup\Gamma_E^+$ has endpoints $x,y$,
\begin{itemize}
	\item If $ Im(\gamma)\subset (X\menos F)\cup E$, then $|f(x)-f(y)|\leq\int_\gamma\rho=\int_\gamma\rho\chi_{X\menos F}$ as $\mu( Im(\gamma)\cap E)=0$.
	\item If $x,y\in F\menos E$, then $f(x)=f(y)=0$ and $|f(x)-f(y)|\leq\int_\gamma\rho\chi_{X\menos F}$ holds trivially.
	\item If $x\in (X\menos F)\cup E$ (or the same for $y$) but $ Im(\gamma)$ is not completely in $(X\menos F)\cup E$, as $(f\circ\gamma_h)^{-1}(\{0\})$ is a closed set of $[0,h(\gamma)]$ ($f\circ\gamma_h$ is continuous), it has a minimum $a$ and maximum $b$ (with $f\circ\gamma_h(a)=f\circ\gamma_h(b)=0$). Then,
	\begin{align*}|f(x)-f(y)|\leq\\
	& \hspace{-1.5cm}\leq|f(x)-f(\gamma_h(a))|+|f(\gamma_h(a))-f(\gamma_h(b))|+|f(\gamma_h(b))-f(y)|\\
	& \hspace{-1.5cm}\leq\int_{\gamma_h|_{[0,a]}}\rho+\int_{\gamma_h|_{[b,h(\gamma)]}}\rho\leq\int_\gamma\rho\chi_{X\menos F}\end{align*} 
	as $\gamma_h([0,a])$ and $\gamma_h([b,h(\gamma)])$ intersect $F$ in a set of $\mu$-measure zero.\end{itemize}\end{proof}

\begin{lem}If $\Gamma^*$ has the $\mu$-arc-chord property with exponent $\beta$, then every Lipschitz-$\beta$ function is absolutely continuous over every curve of $\Gamma^*$.\end{lem}
\begin{proof}Let $\gamma:[0,h]\imply X$ be a path in $\Gamma^*$ parametrized by $\mu$-arc length, and let $f:X\imply\RR$ be Lipschitz with constant $L$. If $\epsilon>0$ and $0\leq a_1<b_1<a_2<b_2<\cdots<a_n<b_n\leq h$ satisfies $\sum_i|b_i-a_i|<\frac{\epsilon}{LC_\mu}$, then
\begin{align*}\sum_i|f(\gamma(b_i))-f(\gamma(a_i))| &\leq L\sum_i d(\gamma(b_i),\gamma(a_i))^\beta\\
&\leq L\sum_i \diam(\gamma([a_i,b_i]))^\beta\\
&\leq LC_\mu\sum_i \mu(\gamma([a_i,b_i]))=LC_\mu\sum_i|b_i-a_i|\\
&<\epsilon.\end{align*}\end{proof}

\begin{lem}If $\Gamma^*$ has the $\mu$-arc-chord property with exponent $\beta$ and $f:X\imply\RR$ is a Lipschitz-$\beta$ function with constant $L$, then $C_\mu L\chi_{supp(f)}$ is an upper gradient of $f$. In particular if $supp(f)$ is compact we have $f\in\tilde{N}^{1,p}$.\end{lem}
\begin{proof}Let $\gamma:[a,b]\imply X$ have endpoints $x,y$. Consider the following cases:
\begin{itemize}
	\item $ Im(\gamma)\subset supp(f)$. Then $|f(x)-f(y)|\leq Ld(x,y)^\beta\leq C_\mu L\mu( Im(\gamma))=\int_\gamma LC=\int_\gamma CL\chi_{supp(f)}$.
	\item $ Im(\gamma)\cap supp(f)=\emptyset$. Then $|f(x)-f(y)|=0=\int_\gamma CL\chi_{supp(f)}$.
	\item $x\in supp(f)$ but $ Im(\gamma)\not\subset supp(f)$. Then as $(f\circ\gamma)^{-1}(\{0\})$ is closed in $[a,b]$, it has minimum $a_0>a$ and maximum $b_0\leq b$. We have that $\gamma([a,a_0])$ and $\gamma([b_0,b])$ are subsets of $supp(f)$ and $f(\gamma(a_0))=f(\gamma(b_0))=0$ so,
	\begin{align*}|f(x)-f(y)|\leq\\
	&\hspace{-1.5cm}\leq|f(x)-f(\gamma(a_0))|+|f(\gamma(a_0))-f(\gamma(b_0))|+|f(\gamma(b_0))-f(y)|\\ 
	&\hspace{-1.5cm}\leq Ld(x,\gamma(a_0))^\beta + Ld(\gamma(b_0),y)^\beta\\
	&\hspace{-1.5cm}\leq LC_\mu\mu(\gamma([a,a_0])) + LC_\mu\mu(\gamma([b_0,b])\\
	&\hspace{-1.5cm}\leq\int_\gamma LC_\mu\chi_{supp(f)}.\end{align*}
\end{itemize}
Finally if $supp(f)$ is compact, $f,CL\chi_{supp(f)}\in L^p(m)$ for every $p$.
\end{proof}

With the previous results, and also requiring the measure $m$ to be doubling, we get the following.

\begin{thm}If $m$ is doubling, $X$ supports a $(1,p)$-Poincaré inequality of exponent $\beta\leq 1$ and $\Gamma^*$ satisfies the $\mu$-arc-chord property with exponent $\beta$, then Lipschitz-$\beta$ functions are dense in $N^{1,p}$.\end{thm}
\begin{proof} Let $f\in\tilde{N}^{1,p}$ and let $g\in L^p$ be an upper gradient of $f$. Assume $f$ is bounded (bounded functions are clearly dense in $N^{1,p}$). We define
\begin{align*}E_k=\{x\in X:Mg^p(x)>k^p\},\end{align*}
where $M$ is the noncentered Hardy-Littlewood maximal function. As $m$ is doubling, $M$ is weak type $1,1$, and
\begin{align*}m(E_k)\leq \frac{C}{k^p}\int_X g^p\imply 0\quad\text{ as }\quad k\imply\infty.\end{align*}
Let $F_k=X\menos E_k$ (which is closed as $E_k$ is open). If $x\in F_k$, $r>0$ and $B=B(x,r)$, 
\begin{align*}\fint_B|f-f_B|\leq C r^\beta(\fint_B g^p)^{1/p}\leq Cr^\beta (Mg^p(x))^{1/p}\leq Cr^\beta k.\end{align*} 
Then if we define $f_n(x)=f_{B(x,2^{-n}r)}$, we have
\begin{align*}|f_{n+j}(x)-f_n(x)| &\leq \sum_{i=1}^j|f_{n+i+1}(x)-f_{n+i}(x)|\\
&\leq \sum_{i=1}^j \fint_{B(x,2^{-(n+i+1)}r)}|f-f_{B(x,2^{-(n+i)}r)}|\\
&\leq C\sum_{i=1}^j\fint_{B(x,2^{-(n+i)}r)}|f-f_{B(x,2^{-(n+i)}r)}|\\
&\leq Ckr^\beta (2^\beta)^{-n}\sum_{i=1}^j 2^{-i}\leq Ckr^\beta 2^{-n\beta},\end{align*}
and therefore $f_n(x)$ is Cauchy for each $x\in F_k$. Now, we define for $x\in F_k$,
\begin{align*}f^k(x)=\lim f_n(x).\end{align*}
Observe that for Lebesgue points of $f$ in $F_k$ we have $f^k(x)=f(x)$. Let's verify that $f^k$ is Lipschitz-$\beta$.
Given $x,y\in F_k$, take $r=d(x,y)$, $B_n=B(x,2^{-n}r)$, $B'_n=B(y,2^{-n}r)$, and
\begin{align*}|f^k(x)-f^k(y)| &\leq\\
&\hspace{-2cm}\leq \sum_{n=0}^\infty |f_n(x)-f_{n+1}(x)| + |f_0(x)-f_0(y)| + \sum_{n=0}^\infty |f_{n}(y)-f_{n+1}(y)|\\
&\hspace{-2cm}\leq \sum_{n=0}^\infty C\fint_{B_n}|f-f_{B_n}| + C\fint_{2B_0}|f-f_{2B_0}| + \sum_{n=0}^\infty C\fint_{B'_n}|f-f_{B'_n}|\\
&\hspace{-2cm}\leq Ckr^\beta\sum_{n=0}^\infty 2^{-n\beta} + Cr^\beta k \leq Ckr^\beta= Ckd(x,y)^\beta.\end{align*}
Now, $f^k$ can be extended to all of $X$ as a Lipschitz-$\beta$ function with the same Lipschitz constant, and we can assume it is bounded by $Ck$ (see \cite{A}). Then
\begin{align*}\int_X|f-f^k|^p=\int_{E_k}|f-f^k|^p\leq C\int_{E_k}|f|^p + Ck^p m(E_k)\imply 0\end{align*}
as $k\imply\infty$, for $m(E_k)\imply 0$ and the weak type of the Hardy-Littlewood maximal implies
\begin{align*}k^pm(E_k)&=k^pm(M(g^p)>k^p)\leq k^p m(M(g^p\chi_{\{g^p>k^p/2\}})>k^p/2)\\
&\leq C\int_{\{g^p>k^p/2\}} g^p\imply 0.\end{align*}

So $f^k$ tends to $f$ in $L^p$. As $f$ y $f^k$ are $ACC_p$, $(g+\tilde{C}k)\chi_{E_k}$ is a $p$-weak upper gradient of $f-f^k$, and as it is in $L^p$ and tends to 0 when $k\imply\infty$, $f-f^k\in N^{1,p}$ for every $k$ and $\|f-f^k\|_{N^{1,p}}\imply0$.\end{proof}

If $X$ is doubling and supports a $(1,q)$ Poincaré inequality of exponent $\beta$ for some $1\leq q<p$, then we have that every function in $N^{1,p}$ has a Haj\symbol{170}asz gradient in $L^p$, i.e. $N^{1,p}\embeds M^{\beta,p}$ with $\|\cdot\|_{M^{\beta,p}}\leq C\|\cdot\|_{N^{1,p}}$ (see \cite{Ha}, \cite{KM}, \cite{Sh}, we define $M^{\beta,p}$ to be the space $M^{1,p}$ for the metric $d^\beta$). The converse embedding holds true in general for Shanmugalingam's case. In our case we need the $\mu$-arc-chord property.

\begin{lem}Assume $\Gamma^*$ satisfies the $\mu$-arc-chord property and let $f$ be a continuous function satisfying
\begin{align*}|f(x)-f(y)|\leq d(x,y)^\beta(g(x)+g(y))\end{align*}
for every $x,y$, for some nonnegative measurable function $g$. Then there exists $C>0$ such that $Cg$ is an upper gradient for $f$.\end{lem}
\begin{proof}Let $\gamma:[0,h]\imply X$ be a path in $\Gamma^*$ parametrized by $\mu$-arc length with endpoints $x,y$. If $\int_\gamma g=\infty$ we are done. Otherwise, for each $n$ we take $\gamma_i=\gamma|_{\left[\frac{i}{n},\frac{i+1}{n}\right]}$, $0\leq i\leq n-1$, as $\gamma$ is a $\mu$-arc length parametrization we have that $\mu(|\gamma_i|)=\mu( Im(\gamma))/n=h/n$. For each $i$, there exists $x_i\in |\gamma_i|$ with
$g(x_i)\leq \fint_{\gamma_i}g$, and the $\mu$-arc-chord property implies that $d(x_i,x_{i+1})^\beta\leq C\mu(|\gamma_i|)$, then
\begin{align*}|f(x_1)-f(x_{n-1})| &\leq\sum_i|f(x_i)-f(x_{i+1})|\\
&\leq \sum_i d(x_i,x_{i+1})^\beta(g(x_i)+g(x_{i+1}))\\
&\leq C\sum_i \left(\int_{\gamma_i}g + \int_{\gamma_{i+1}}g\right)\\
&\leq C\int_\gamma g.\end{align*}
Taking $n\imply\infty$, $x_0\imply x, x_{n-1}\imply y$ and
\begin{align*}|f(x)-f(y)|\leq C\int_\gamma g\end{align*}
and we have what we needed.\end{proof}

\begin{coro}If $\Gamma^*$ satisfies the $\mu$-arc-chord property  with exponent $\beta$ and continuous functions are dense in $M^{\beta,p}$ (which happens for instance if $\beta\leq 1$), then $M^{1,p}\embeds N^{1,p}$, with $\|\cdot\|_{N_{1,p}}\leq C\|\cdot\|_{M^{\beta,p}}$.\end{coro}

\begin{thm}\label{mp} If $X$ is doubling and supports a $(1,q)$ Poincaré inequality with exponent $\beta\leq 1$ for some $1\leq q<p$, and $\Gamma^*$ satisfies the $\mu$-arc-chord property with exponent $\beta$, then $M^{1,p}=N^{1,p}$, with equivalent norms.\end{thm}

As in \cite{Sh}, we have the following versions of the classical Sobolev embedding theorems. In Shanmugalingam's case they are proven for $\beta=1$, but the same proof can be applied for other $\beta$ in our case.

\begin{thm}\label{sob1}If $m$ is doubling and satisfies
\begin{align*}m(B(x,r))\geq Cr^N\end{align*}
for $C,N$ independent of $x\in X, 0<r<2\diam (X)$, and if $X$ supports a $(1,p)$ Poincaré inequality of exponent $\beta\leq 1$ for $p>N/\beta$, then functions in $N^{1,p}$ are Lipschitz-$\alpha$ with $\alpha=\beta-N/p$.\end{thm}

\begin{thm}\label{sob2}If $X$ is bounded and satisfies
\begin{align*}cr^N\leq m(B(x,r))\leq Cr^N\end{align*}
with $c,C,N$ independent of $x\in X, 0<r<2\diam (X)$ (i.e. $X$ is Ahlfors $N$-regular), and if $X$ supports a $(1,q)$ Poincaré inequality of exponent $\beta$ for $q>1/\beta$, then for $p$ satisfying $q<p<Nq$, $\frac{1}{p^*}=\frac{1}{p}-\frac{1}{Nq}$ we have that every $f\in N^{1,p}$ with upper gradient $g$,
\begin{align*}\|u-u_X\|_{p^*}\leq C \diam(X)^{\beta-1/q}\|g\|_p.\end{align*}\end{thm}

We finish this work with the example $X=\RR^n$ with Euclidean distance, $d\mu=\omega d\HH^1$, $dm=\omega^pdx$ where $\omega$ and $\frac{1}{\omega}$ are locally integrable. First we consider when a Poincaré inequality holds.

If $\omega$ is bounded, as Poincaré inequality is true for $dx$, we get
\begin{align*}\fint_B|f-f_B|dm &\leq \fint_B|f-f_{B,dx}|dm+|f_B-f_{B,dx}|\\
&\leq 2\fint_B|f-f_{B,dx}|dm\\
&\leq 2\left(\fint_B|f-f_{B,dx}|^pdm\right)^{1/p}\\
& \leq 2\left(\frac{|B|}{m(B)}\right)^{1/p}\left(\frac{1}{|B|}\int_B |f-f_{B,dx}|^p\omega^pdx\right)^{1/p}\\
&\leq C\left(\frac{|B|}{m(B)}\right)^{1/p}\|\omega\|_\infty \diam(B)\left(\frac{1}{|B|}\int_B|\nabla f|^p dx\right)^{1/p}\\
&= C\|\omega\|_\infty\diam(B)\left(\fint_B\left(\frac{|\nabla f|}{\omega}\right)^p dm\right)^{1/p},\end{align*}
where $f_{B,dx}=\fint_B fdx$.

Instead of asking for $\omega$ to be bounded, we may use a two-weight Poincaré inequality as found in \cite{Hr}. Let $1<p<n$, $\omega^p\in A_\infty$ and
\begin{align*}\frac{1}{|Q|^{q(\frac{1}{p}-\frac{1}{n})}}\int_Q \omega^pdx\leq C\end{align*}
for each cube $Q$, with $C$ independent of $Q$, and some $q$ such that $\frac{1}{p}-\frac{1}{n}\leq \frac{1}{q}<\frac{1}{p}$. If $p=q=2$ this would be Fefferman-Phong's condition (see \cite{FP}). 

In our case, the pair $1,\omega^p$ satisfies condition $A^{1/n}_{p,q}$, where we say two weights $w_1,w_2$ satisfy condition $A^\alpha_{p,q}$ if there exists $C>0$ such that
%\begin{align*}\left(\int_Q M_\alpha(\chi_Q w_1^{-p'/p})(x)^qw_2(x)dx\right)^{1/q}\leq C\left(\int_Q w_1(x)^{-p'/p}dx\right)^{1/p}\end{align*}
\begin{align*}\left(\int_Q w_1^{-p'/p}\right)^{1/p'}\left(\int_Q w_2\right)^{1/q}\leq C|Q|^{1-\alpha}\end{align*}
for each cube $Q$, for $0\leq\alpha<1$, $1<p,q<\infty$, $1/p-\alpha\leq 1/q$.
%, where $M_\alpha$ is the fractional maximal function,
%\begin{align*}M_\alpha g(x)=\sup_{x\in B}\frac{1}{|B|^{1-\alpha}}\int_B g;\end{align*}
%in our case this maximal function satisfies, for $x\in Q$, $M_\alpha(\chi_Q)(x)\sim |Q|^{1/n}$, so condition $A^{1/n}_{p,p}$ is equivalent to
%\begin{align*}\frac{1}{|Q|^{1-p/n}}\int_Q \omega^pdx\leq C,\end{align*}

E. Harboure proves in \cite{Hr} that these conditions imply there exists constants $C>0$ and $\delta>0$ (depending on the $A_\infty$ and $A^{1/n}_{p,q}$ constants) such that the following Poincaré inequality holds
\begin{align*}\int_Q |f-f_{Q,dx}|^p\omega^p dx\leq C \left(\int_Q \omega^pdx\right)^\delta \int_Q|\nabla f|^pdx.\end{align*}
From this condition, our $(1,p)$ Poincaré inequality follows,
\begin{align*}\fint_Q|f-f_Q|dm &\leq C\frac{1}{m(Q)^{1/p}}\left(\int_Q |f-f_{Q,dx}|^p\omega^pdx\right)^{1/p}\\
& \leq C\frac{1}{m(Q)^{1/p}}\left(m(Q)^\delta \int_Q \left(\frac{|\nabla f|}{\omega}\right)^pdm\right)^{1/p}\\
& = C m(Q)^{\delta/p}\left(\fint_Q \left(\frac{|\nabla f|}{\omega}\right)^pdm\right)^{1/p}\\
& \leq C\diam(Q)^\beta \left(\fint_Q \left(\frac{|\nabla f|}{\omega}\right)^pdm\right)^{1/p},
\end{align*}
for $\beta=\frac{\delta q}{p}\left(\frac{n}{p}-1\right)$, where the last inequality follows from the fact that, by our assumption, as $dm=\omega^pdx$, 
\begin{align*}m(Q)=\int_Q\omega^pdx\leq C|Q|^{q\left(\frac{1}{p}-\frac{1}{n}\right)}=C\diam(Q)^{q(n/p-1)}.\end{align*}

As an example of such $\omega$, we may consider $\omega(x)=\frac{1}{|x|^\lambda}$, for some $0\leq \lambda<1$. Then $\omega^p\in A_\infty$ if $p\lambda <n$ and the pair $1,\omega^p$ satisfies condition $A^{1/n}_{p,q}$ for $q=\frac{n-\lambda p}{n-p}p$: for $Q=Q(0,R)$,
\begin{align*}\frac{1}{|Q|^{q(1/p-1/n)}}\int_Q \omega^pdx=CR^{-q\frac{n-p}{p}}\int_{Q(0,R)}\frac{1}{|x|^{\lambda p}}dx\sim R^{-q\frac{n-p}{p}}R^{n-\lambda p}=C\end{align*}
and for $Q=Q(x_0,R)$ with $x_0\neq 0$, we consider two cases. If $2R>|x_0|$, then $Q(x_0,R)\subset Q(0,3R)$, so
\begin{align*}\frac{1}{|Q|^{q(1/p-1/n)}}\int_Q \omega^pdx\leq CR^{-q\frac{n-p}{p}}\int_{Q(0,3R)}\frac{1}{|x|^{\lambda p}}dx\leq C;\end{align*}
on the other hand if $2R\leq |x_0|$, then for $x\in Q$ we have $|x|\sim |x_0|$, so
\begin{align*}\frac{1}{|Q|^{q(1/p-1/n)}}\int_Q \omega^pdx\sim R^{-q\frac{n-p}{p}}\frac{1}{|x_0|^{\lambda p}}R^n\leq C.\end{align*}

As a special case, we can consider $\lambda=0$, so the weight $\omega=1$, which gives classical Sobolev spaces $W^{1,p}$, is included in our result.

With Poincaré inequality, theorems \ref{sob1} and \ref{sob2} hold, provided the other conditions are met. We also obtain one half of theorem \ref{mp}, as a Poincaré inequality is sufficient to obtain $N^{1,p}\embeds M^{\beta,p}$.
%
%We can also consider only cubes $Q\subset Q_0$ for some fixed $Q_0$, and consider condition $S^{1/n}_{p,p}$ restricted to those cubes. This case includes as a particular case the weight $\omega=1$.

If there exists $c>0$ such that $\omega(x)\geq c$ for all $x$, we also get the arc-chord property,
\begin{align*}\diam(Im(\gamma)) &\leq \HH^1(Im(\gamma))=\int_{Im(\gamma)}d\HH^1\\
&\leq \frac{1}{c}\int_{Im(\gamma)}\omega d\HH^1=\frac{1}{c}\mu(Im(\gamma)).\end{align*}

For example, the weight $\omega(x)=\frac{1}{|x|^\lambda}$ satisfies this restriction if $\lambda=0$ or if $X=Q_0$ for some fixed cube $Q_0$, here we consider only cubes $Q\subset Q_0$ (that may contain the origin, so $\omega$ is not necessarily bounded), and as it also satisfies the $A^{1/n}_{p,q}$ condition restricted to those cubes. This case allows for both a Poincaré inequality and an arc-chord property, even though the exponents in each case may not coincide.

\textit{E-mail address:} \texttt{mmarcos@santafe-conicet.gov.ar}

Instituto de Matemática Aplicada del Litoral, CONICET, UNL.

CCT CONICET Santa Fe, Predio ``Alberto Cassano'', Colectora Ruta Nac. 168 km 0, Paraje El Pozo, 3000 Santa Fe, Argentina.

\end{document}